%% file: heat_kernel_for_open_manifolds.v2.tex
\documentclass[12pt]{article}

\include{headers}					
\usepackage{natbib}					

\title{Heat Kernel for Open Manifolds}
\author{Trevor H. Jones}
\date{\today}

\newtheorem{theorem}{Theorem}
\newtheorem{lemma}[theorem]{Lemma}
\newtheorem{corollary}[theorem]{Corollary}

\begin{document}
	\maketitle
	\begin{abstract}
		In a 1991 paper by Buttig and Eichhorn, the existence and uniqueness of a differential forms heat kernel on open manifolds of bounded geometry was proven.  In that paper, it was shown that the heat kernel obeyed certain properties, one of which was a relationship between the derivative of heat kernel of different degrees.  We will give a proof of this condition for complete manifolds with Ricci curvature bounded below, and then use it to give an integral representation of the heat kernel of degree $k$.
	\end{abstract}

		In this paper we are considering the differential forms heat equation on manifolds, in particular we are considering $(\lp + \di_t) \omega = 0$ with Dirchlet initial conditions.  Our goal is to produce a formula for the Green's function, also known as the heat kernel or fundamental solution, which gives the solution of this equation.

		The solutions of this equation in the case of functions, or $0$-forms, is well-known.  The work on differential forms has been much more recent.  In 1983, Dodziuk \cite{mpfpiathfoom} proved that for complete oriented $C^\infty$ Riemannian manifolds with Ricci curvature bounded below, bounded solutions are uniquely determined by their initial values.  In 1991 Buttig and Eichhorn \cite{thkfpomobg} gave existence and uniqueness results for the differential forms heat kernel on open manifolds of bounded geometry.  In the same paper they also gave an identity relating derivatives of the heat kernel of differing degrees.  Using that identity, the author, in his doctoral thesis, \cite{thkonrs}, was able to give a formula for the $1$-form heat kernel on open Riemann surfaces of bounded geometry.  This formula,
		\begin{equation} \label{eq:thjphd}
			K_1(\vx{x}, \vx{y} , t) = \left(I + *_\vx{x} *_\vx{y} \right) d_\vx{x} d_\vx{y} \int_t^\infty \!\!\! K_0(\vx{x}, \vx{y}, \tau) \dx{\tau}{},
		\end{equation}
		directly relates the $1$-form heat kernel to the $0$-form heat kernel, about which more is known.

		One of the properties given by Buttig and Eichhorn for a global heat kernel was that the heat kernels $K_k(\vx{x}, \vx{y}, t)$ and $K_{k+1}(\vx{x}, \vx{y}, t)$ are related by
			\begin{equation} \label{defitem:b_and_e_cond}
				d_\vx{x} K_k(\vx{x}, \vx{y}, t) = d^*_\vx{y} K_{k+1}(\vx{x}, \vx{y}, t).
			\end{equation}
		In this article, we will present an alternate proof of this property for manifolds with Ricci curvature bounded below, and then use this to give a formula for the $k$-form heat kernel.

		\begin{lemma} \label{lem:b_and_e}
			For a complete manifold $M$, with Ricci curvature bounded from below, we have the following relationship between the $k$- and $(k+1)$-form heat kernels:
			\begin{enumerate}
				\item $d_\vx{x} K_k(\vx{x}, \vx{y}, t) = d^*_\vx{y} K_{k+1}(\vx{x}, \vx{y}, t)$ \label{lemitem:b_and_e}
				\item $d_\vx{y} K_k(\vx{x}, \vx{y}, t) = d^*_\vx{x} K_{k+1}(\vx{x}, \vx{y}, t)$ \label{lemitem:b_and_e_sym}
			\end{enumerate}
		\end{lemma}

		\begin{proof}
			Let $E(\vx{x}, \vx{y}, t) = d_\vx{x} K_k(\vx{x},\vx{y},t) - d_\vx{y}^* K_{k+1}(\vx{x},\vx{y},t).$  We will demonstrate that $E$ satisfies the heat equation with zero as the initial condition.  This will imply, by uniqueness of the solutions of the heat equation, see \cite{mpfpiathfoom}, that $E \equiv 0$, giving the desired result.

			First
			\begin{align*}
				\lp_\vx{x} E & = \lp_\vx{x} d_\vx{x} K_k(\vx{x},\vx{y},t) - \lp_\vx{x} d_\vx{y}^* K_{k+1}(\vx{x},\vx{y},t) \\
				& = d_\vx{x} \lp_\vx{x} K_k(\vx{x},\vx{y},t) - d_\vx{y}^* \lp_\vx{x} K_{k+1}(\vx{x},\vx{y},t) \\
				& = d_\vx{x} (-\di_t) K_k(\vx{x},\vx{y},t) - d_\vx{y}^* (-\di_t) K_{k+1}(\vx{x},\vx{y},t) \\
				& = -\di_t E.
			\end{align*}

			Next consider $W := \ip{E}{\omega(\vx{x})}$, where $\omega$ is a suitable test function and $\ip{\mu}{\nu} = \int_M \mu \wedge * \nu$.  Then
			\begin{align*}
				\lim_{t \rightarrow 0+} W & = \lim_{t \rightarrow 0+} \ip{d_\vx{x} K_k}{\omega(\vx{x})} - \ip{d_\vx{y}^* K_{k+1}}{\omega(\vx{x})} \\
				& = \lim_{t \rightarrow 0+} \ip{K_k}{d_\vx{x}^* \omega(\vx{x})} - d_\vx{y}^* \ip{K_{k+1}}{\omega(\vx{x})} \\
				& = d_\vx{y}^* \omega(\vx{y}) - d_\vx{y}^* \omega(\vx{y}) = 0
			\end{align*}

			Since $\omega$ was an arbitrary test function, we must have that $E \equiv 0$ at $t=0$.  Thus by uniqueness, $E \equiv 0$ for all $t>0$.

			The proof of the second assertion follows in a similar manner.
		\end{proof}

		We will use this result to give an explicit formula for $K_k$ in terms of $K_{k \pm 1}$.

		\begin{theorem} \label{thm:K_k}
			Let $M$ be an open, complete manifold with Ricci curvature bounded below.  Then the differential forms heat kernel obey the following relation:
				$$K_k(\vx{x},\vx{y}, t) = d_\vx{x} d_\vx{y} \int_t^\infty K_{k-1}(\vx{x},\vx{y}, \tau) \dx{\tau}{} + d^*_\vx{x} d^*_\vx{y} \int_t^\infty K_{k+1}(\vx{x},\vx{y}, \tau) \dx{\tau}{}.$$
		\end{theorem}

		\begin{proof}
			Let $K_k$ be the $k$-form heat kernel.  Clearly,
				$$\ds K_k(\vx{x}, \vx{y}, t) = -\int_t^\infty \frac{\di}{\di \tau} K_k(\vx{x}, \vx{y}, \tau) \dx{\tau}{},$$
			since $K_k$ tends to zero (pointwise) as $t$ increases.  Since $K_k$ is a solution of the heat equation, we can replace the time derivative with $-\lp_\vx{x} = - d_\vx{x} d^*_\vx{x} - d^*_\vx{x}d_\vx{x}$, so
				$$K_k(\vx{x}, \vx{y}, t) = \int_t^\infty (d_\vx{x} d^*_\vx{x} + d^*_\vx{x}d_\vx{x}) K_k(\vx{x}, \vx{y}, \tau) \dx{\tau}{}.$$
			Using Lemma \ref{lem:b_and_e}, we can rewrite the above as
				$$K_k(\vx{x}, \vx{y}, t) = \int_t^\infty d_\vx{x} d_\vx{y}K_{k-1}(\vx{x}, \vx{y}, \tau) + d^*_\vx{x} d^*_\vx{y} K_{k+1}(\vx{x}, \vx{y}, \tau) \dx{\tau}{}.$$
		\end{proof}

		The result in Theorem \ref{thm:K_k} depends mainly on two things: the existence and uniqueness of the heat kernel, and the pointwise convergence to zero of the kernel for large time.  The methods used above work for a diffusion-type equation provided these conditions are met.  For example, for the diffusion equation $(\lp + c \di_t) \omega = 0$, the proofs follow through almost identically.

		\begin{corollary} \label{cor:scale_time}
			Let $M$ be an open, $n$-dimensional, differentiable manifold, with Ricci curvature bounded below, and consider the differential form diffusion equation $(\lp + c \di_t) \omega = 0$ with initial data $\omega(\vx{x},0) = f(\vx{x})$.  Then the Green's functions are related by
			$$G_k(\vx{x},\vx{y},t) = d_\vx{x} d_\vx{y} \int_{ct}^\infty G_{k-1}(\vx{x},\vx{y}, \tau) \dx{\tau}{} + d^*_\vx{x} d^*_\vx{y} \int_{ct}^\infty G_{k+1}(\vx{x},\vx{y}, \tau) \dx{\tau}{}.$$
		\end{corollary}

		\begin{proof}
			Let $T=ct$, then $c \di_t = \di_T$, so the equation becomes $\left( \lp + \di_T \right) \omega(\vx{x},T) = 0$ with the same initial conditions.  So by Theorem \ref{thm:K_k} we have the desired Green's functions.
		\end{proof}

		In the case of 2-dimensional manifolds, the $0$-form and the $2$-form heat kernels are isomorphic, as the following Lemma will show.  This allows use to write the $1$-form heat kernel in terms of the $0$-form, or function, heat kernel.

		\begin{lemma} \label{lem:kssk}
			Let $M$ be a complete manifold with Ricci curvature bounded below.  Then the differential forms heat kernels, $K_k$ and $K_{n-k}$ are related in the following manner:
				$$K_k = *_\vx{x}*_\vx{y} K_{n-k}.$$
		\end{lemma}

		\begin{proof}
			Consider the equation $(\di_t + \lp_k)u = 0, \; \; u(\vx{x},0) = f(\vx{x})$.  Then $u$ is given by $\ds u(\vx{x},t) = \ip{K_k(\vx{x}, \vx{y},t)}{f(\vx{y})} = \int_M K(\vx{x},\vx{y},t) \wedge_\vx{y} *_\vx{y} f(\vx{y})$.  Since $* \lp_k = \lp_{n-k}*$, it follows that $*_\vx{x} u$ is a solution of the $(n-k)$-form heat equation with initial condition $*_\vx{x} f(\vx{x})$.  So
			\begin{align*}
				*_\vx{x} u(\vx{x},t) & = \ip{K_{n-k}(\vx{x},\vx{y},t)}{*_\vx{y} f(\vx{y})} \\
				& = \ip{*_\vx{y} f(\vx{y})}{K_{n-k}(\vx{x},\vx{y},t)} \\
				& = \int_M *_\vx{y} f(\vx{y}) \wedge_\vx{y} *_\vx{y} K_{n-k}(\vx{x},\vx{y},t)
			\end{align*}
			By applying $*_\vx{x}$ to both sides, and changing order in the wedge product, we have
				$$(-1)^{k(n-k)} u(\vx{x},t) = \int_M (-1)^{k(n-k)} *_\vx{x} *_\vx{y} K_{n-k}(\vx{x}, \vx{y}, t) \wedge_\vx{y} *_\vx{y} f(\vx{y})$$
			or
				$$u(\vx{x},t) = \ip{*_\vx{x} *_\vx{y} K_{n-k}(\vx{x}, \vx{y}, t)}{f(\vx{y})}.$$
			By uniqueness of the heat kernel we have the desired result.
		\end{proof}

		\begin{corollary} \label{cor:K1}
			Let $M$ be an open, complete manifold of dimension 2 with Ricci curvature bounded below.  Then the 1-form heat kernel on $M$ is given by
				$$K_1(\vx{x}, \vx{y}, t) = \left( I + *_\vx{x} *_\vx{y} \right) d_\vx{x} d_\vx{y} \int_t^\infty K_0(\vx{x}, \vx{y}, \tau) \dx{\tau}{}$$
			where, $\vx{x}, \vx{y} \in M$ and $t>0$ and $K_0$ is the 0-form heat kernel.
		\end{corollary}

		\begin{proof}
			Since $M$ has dimension 2, and so $K_2 = *_\vx{x} *_\vx{y} K_0$ by Lemma \ref{lem:kssk}.  Recall that $d^* * \omega = - * d \omega$ for $0$-forms, $\omega$.  This gives the desired result.
		\end{proof}

		As an example, consider the case of the hyperbolic plane, with constant curvature $-1$.  From \cite{eirg} we have the $0$-form heat kernel
			$$K_0(\vx{x},\vx{y},t) = \frac{1}{2 \pi} \int_0^\infty P_{-\frac{1}{2}+ i\rho}(\cosh d_{H^2}(\vx{x},\vx{y})) \rho e^{-(\frac{1}{4} + \rho^2)t} \tanh \pi \rho d \rho,$$
		which, if we perform the integration set out in Corollary \ref{cor:K1}, we get
			\begin{align*}
				{\textstyle K_1(\vx{x}, \vx{y}, t) = \frac{1}{2 \pi} \left( I +  *_\vx{x} *_\vx{y} \right) d_\vx{x} d_\vx{y} } \left[ \int_0^\infty {\scriptstyle P_{-\frac{1}{2}+ i\rho}(\cosh d_{H^2}(\vx{x},\vx{y})) \rho \frac{e^{-(\frac{1}{4} + \rho^2)t}}{\frac{1}{4} + \rho^2} \tanh \pi \rho \dx{\rho}{} } \right].
			\end{align*}

		If $M$ is an open $2$-dimensional manifold which has a unique heat kernel for functions, $K_0$, then Corollary \ref{cor:K1} suggests a candidate for a heat kernel on $1$-forms, and since the $K_0$ and $K_2$ heat kernels are isomorphic, we would know all the heat kernels.  We will now show that $K_1(\vx{x},\vx{y},t)=(I+*_\vx{x}*_\vx{y})d_\vx{x}d_\vx{y} \int_t^\infty K_0(\vx{x},\vx{y},\tau) d\tau$ works as the heat kernel.  So given
			\begin{align}
				(\lp^{(1)}_\vx{x} + \di_t)w_1(\vx{x},t) & = 0 \label{eq:heK1.1} \\
				w_1(\vx{x},0) & = f_1(\vx{x}) \label{eq:heK1.2}
			\end{align}
		show that $w_1$ can be written as $w_1(\vx{x},t) = \ip{K_1(\vx{x},\vx{y},t)}{f_1(\vx{y})}$.

		So, let $w_1$ be a solution of (\ref{eq:heK1.1}) and (\ref{eq:heK1.2}), and $w(\vx{x},t) = \ip{K_1(\vx{x},\vx{y},t)}{f_1(\vx{y})}$, Since the Laplacian commutes with the Hodge star isomorphism and the exterior derivative and coderivative, it is clear that $w$ satisfies equation (\ref{eq:heK1.1}).  Now we just need to show that $w$ as defined, satisfies the initial condition (\ref{eq:heK1.2}).

		\begin{align*}
			w(\vx{x},t) & = \int_t^\infty d_\vx{x} \ip{d_\vx{y} K_0(\vx{x},\vx{y},\tau)}{f_1(\vx{y})} + *_\vx{x} d_\vx{x} \ip{*_\vx{y} d_\vx{y} K_0(\vx{x},\vx{y},\tau)}{f_1(\vx{y})} d\tau \\
			& = \int_t^\infty d_\vx{x} \ip{K_0(\vx{x},\vx{y},\tau)}{d_\vx{y}^* f_1(\vx{y})} - *_\vx{x} d_\vx{x} \ip{K_0(\vx{x},\vx{y},\tau)}{*_\vx{y} d_\vx{y} f_1(\vx{y})} d\tau \\
			& = \int_t^\infty d_\vx{x} d_\vx{x}^* w_1(\vx{x},\tau) - *_\vx{x} d_\vx{x} *_\vx{x} d_\vx{x} w_1(\vx{x},\tau) d\tau \\
			& = \int_t^\infty \lp w_1(\vx{x}, \tau) d\tau \\
			& = \int_t^\infty -\di_\tau w_1(\vx{x}, \tau) d\tau = w_1(\vx{x},t)
		\end{align*}

		Since $w_1$ is a solution of the heat equation with initial value $f_1$, and $w = w_1$, this means that $w$ also has initial value $f_1$.  Thus $w$ is a solution of (\ref{eq:heK1.1}) and (\ref{eq:heK1.2}).

		Finally, let us consider the case of compact complete manifolds.  In this case, because of conservation, diffusion does not tend to zero, so the large-time limit has to be taken into account.

		\begin{theorem}
			Let $M$ be a complete manifold with Ricci curvature bounded below, and that the $\lim_{t \rightarrow \infty} K_k(\vx{x}, \vx{y},t)$ is a constant double-form, call it $C$.  Then, the heat kernel obeys the following relation:
				$$K_k(\vx{x},\vx{y}, t) = C + d_\vx{x} d_\vx{y} \int_t^\infty K_{k-1}(\vx{x},\vx{y}, \tau) \dx{\tau}{} + d^*_\vx{x} d^*_\vx{y} \int_t^\infty K_{k+1}(\vx{x},\vx{y}, \tau) \dx{\tau}{}.$$
		\end{theorem}

		\begin{proof}
			Let $K_k$ be the $k$-form heat kernel.  Clearly,
				$$\ds K_k(\vx{x}, \vx{y}, t) = C - \int_t^\infty \frac{\di}{\di \tau} K_k(\vx{x}, \vx{y}, \tau) \dx{\tau}{},$$
			since $K_k$ tends to $C$ as $t$ increases.  Since $K_k$ is a solution of the heat equation, we can replace the time derivative with $-\lp_\vx{x} = - d_\vx{x} d^*_\vx{x} - d^*_\vx{x}d_\vx{x}$, so
				$$K_k(\vx{x}, \vx{y}, t) = C + \int_t^\infty (d_\vx{x} d^*_\vx{x} + d^*_\vx{x}d_\vx{x}) K_k(\vx{x}, \vx{y}, \tau) \dx{\tau}{}.$$
			Using Lemma \ref{lem:b_and_e}, we can rewrite the above as
				$$K_k(\vx{x}, \vx{y}, t) = C + \int_t^\infty d_\vx{x} d_\vx{y}K_{k-1}(\vx{x}, \vx{y}, \tau) + d^*_\vx{x} d^*_\vx{y} K_{k+1}(\vx{x}, \vx{y}, \tau) \dx{\tau}{}.$$
		\end{proof}

		I would like to thank Dr. Kucerovsky for his helpful comments of earlier versions of this paper.
	\bibliography{articles}
	\bibliographystyle{alpha}
\end{document}

%% file: headers.tex
\usepackage{amsfonts}
\usepackage{amssymb}
\usepackage{amsmath}
\usepackage{latexsym}
\usepackage{ifthen}
\usepackage[dvips]{color}

\newcommand{\ip}[2]{\left\langle {#1}, {#2} \right\rangle}

\newcommand{\di}{\partial}

\newcommand{\lp}{\Delta}

\newcommand{\vx}[1]{{\bf {#1}}}
\newcommand{\dx}[2]{d{#1}^{#2}}

\newcommand{\ds}{\displaystyle}

\newtheorem{defi}{Definition}[section]

\newtheorem{ex}[defi]{Example}
\newtheorem{remar}[defi]{Remark}

\newenvironment{proof}{\noindent {\bf Proof:}}{\hfill{$\blacksquare$} \newline}

\makeatletter
\@addtoreset{equation}{section}
\@addtoreset{figure}{section}
\@addtoreset{table}{section}
\makeatother

%% file: heat_kernel_for_open_manifolds.v2.bbl
\begin{thebibliography}{Dod83}

\bibitem[BE91]{thkfpomobg}
Ingolf Buttig and Jurgen Eichhorn.
\newblock The heat kernel for $p$-forms on manifolds of bounded geometry.
\newblock {\em Acta Sci. Math.}, 55:33--51, 1991.

\bibitem[Cha84]{eirg}
Isaac Chavel.
\newblock {\em Eigenvalues in Riemannian Geometry}, volume 115 of {\em Pure and
  Applied Mathematics}.
\newblock Academic Press, 1984.

\bibitem[Dod83]{mpfpiathfoom}
Jozef Dodziuk.
\newblock Maximum {P}rinciple for {P}arabolic {I}nequalities and the {H}eat
  {F}low on {O}pen {M}anifolds.
\newblock {\em Indiana University Mathematics Journal}, 32(5):703--16, 1983.

\bibitem[Jon08]{thkonrs}
Trevor~H. Jones.
\newblock {\em The Heat Kernel on Noncompact Riemann Surfaces}.
\newblock PhD thesis, University of New Brunswick, 2008.

\end{thebibliography}
